\documentclass[english,oneside,a4paper,11pt]{article}
\usepackage{amssymb}
\usepackage{latexsym}
\usepackage[english]{babel}
\usepackage{amsmath}
\usepackage{amsfonts}
\usepackage{amsbsy}
\usepackage[mathscr]{eucal}
\usepackage[latin1]{inputenc}
\usepackage{verbatim}
\usepackage{mathrsfs}
\usepackage{graphicx}
\usepackage{float}
\usepackage{bbm}
\usepackage{lscape}
\usepackage{mathdots}
\usepackage{empheq}
\usepackage[dvipsnames]{xcolor}
\usepackage{array}
\usepackage{a4wide}
\usepackage{enumitem}
\usepackage[dvips,colorlinks,bookmarks,breaklinks]{hyperref}  
\hypersetup{linkcolor=black,citecolor=black,filecolor=black,urlcolor=black}
\usepackage{slashbox,pict2e}

\providecommand{\abs}[1]{\left\lvert#1 \right\rvert}
\providecommand{\Log}[1]{\mathrm{Log} #1}
\providecommand{\LLog}[1]{\mathrm{Log \, Log} #1}

\newtheorem{theorem}{Theorem}
\newtheorem{lemma}{Lemma}
\newtheorem{corollary}{Corollary}

\newenvironment{proof}
{\begin{trivlist}\item[\hskip%
\labelsep{{\it \noindent Proof.}}]}{\hfill $\square$
\end{trivlist}}

\newcounter{counter}
\newcommand{\counter}{\stepcounter{counter}\thecounter}

\newenvironment{remark}
{\begin{trivlist}\item[\hskip%
\labelsep{{\it \noindent Remark \counter}}]}{\hfill
\end{trivlist}}

\newenvironment{proofTheorem1}
{\begin{trivlist}\item[\hskip%
\labelsep{{\it \noindent Proof of Theorem 1.}}]}{\hfill $\square$
\end{trivlist}}

\newenvironment{proofCorollary}
{\begin{trivlist}\item[\hskip%
\labelsep{{\it \noindent Proof of Corollary 1.}}]}{\hfill $\square$
\end{trivlist}}

\newenvironment{acknowledgements}
{\begin{trivlist}\item[\hskip%
\labelsep{\bf \noindent Acknowledgements.}]}{\hfill
\end{trivlist}}

\numberwithin{equation}{section}

\allowdisplaybreaks

\begin{document}
\begin{center}
{\huge \textbf{On the rates of convergence for \\ sums of dependent random variables}} \\[25pt]
{\Large Jo\~{a}o Lita da Silva\footnote{\textit{E-mail address:} \texttt{jfls@fct.unl.pt; joao.lita@gmail.com}}} \\
\vspace{0.1cm}
\textit{Department of Mathematics and GeoBioTec \\ Faculty of Sciences and Technology \\
NOVA University of Lisbon \\ Quinta da Torre, 2829-516 Caparica,
Portugal}
\end{center}


\bigskip

\begin{abstract}
For a sequence $\{X_{n}, \, n \geqslant 1 \}$ of nonnegative random variables where $\max[\min(X_{n} - s,t),0]$, $t > s \geqslant 0$, satisfy a moment inequality, sufficient conditions are given under which $\sum_{k=1}^{n} (X_{k} - \mathbb{E} \, X_{k})/b_{n} \overset{\textnormal{a.s.}}{\longrightarrow} 0$. Our statement allows us to obtain a strong law of large numbers for sequences of pairwise negatively quadrant dependent random variables under sharp normalising constants.
\end{abstract}

\bigskip


\bigskip


\bigskip

\section{Introduction}

\indent

The famous Marcinkiewicz-Zygmund strong law of large numbers states that if $\{X_{n}, \, n \geqslant 1 \}$ is a sequence of independent and identically distributed random variables then, for any $0 < p < 2$, $(\sum_{k=1}^{n} X_{k} - nc)/n^{1/p} \overset{\textnormal{a.s.}}{\longrightarrow} 0$ for some finite constant $c$ if and only if $\mathbb{E} \lvert X_{1} \rvert^{p} < \infty$, and if so, $c = \mathbb{E} \, X_{1}$ when $1 \leqslant p < 2$ while $c$ is arbitrary (and hence may be taken as zero) for $0 < p < 1$ (see \cite{Chow97}, page $125$). It should be pointed out that the key ingredient to establish this notable statement relies in the use of maximal inequalities (namely, the well-known L\'{e}vy inequalities). In 1981 and discarding maximal inequalities, Etemadi took advantage of the monotonicity of sums of nonnegative random variables, showing that if $\{X_{n}, \, n \geqslant 1 \}$ is a sequence of pairwise independent, and identically distributed random variables such that $\mathbb{E} \lvert X_{1} \rvert < \infty$ then $\sum_{k=1}^{n} X_{k}/n  \overset{\textnormal{a.s.}}{\longrightarrow} \mathbb{E} \, X_{1}$ (see \cite{Etemadi81}). Later on, Etemadi's argument was extended to general settings by many other authors (see, for instance, \cite{Chandra92}, \cite{Chandra93}, \cite{Chen16}, \cite{Csorgo83}, \cite{Petrov08} or \cite{Walk05}). Meanwhile, assuming $n^{1/p}$ $(1 < p < 2)$ as normalising constants, the study of moment conditions for sequences of pairwise independent, and identically distributed random variables has proceed; for this, we enhance \cite{Choi85}, \cite{Martikainen95} and specially \cite{Sung14}. Indeed, by combining both maximal moment inequality and monotonicity of sums of nonnegative random variables, Sung proved in \cite{Sung14} that $\mathbb{E} \lvert X_{1} \rvert^{p} (\mathrm{Ln} \, \mathrm{Ln} \, \lvert X_{1} \rvert)^{2(p - 1)} < \infty$ (here, $\mathrm{Ln} \, x := \max\{1, \ln x\}$) is a sufficient condition to obtain $\sum_{k=1}^{n} (X_{k} - \mathbb{E} \, X_{k})/n^{1/p} \overset{\textnormal{a.s.}}{\longrightarrow} 0$.

Nowadays, it is not known if original Marcinkiewicz-Zygmund strong law of large numbers for $1 < p < 2$ can be, or can not be, announced to sequences of pairwise independent, and identically distributed random variables. The goal of this paper is to give a contribution for this quest by exploring Sung's technique to establish sharped normalising constants on one hand, and to enlarge the class of random variables, on the other.

We shall need to introduce some relevant notations which will be employed along this paper. Associated to a probability space $(\Omega, \mathcal{F}, \mathbb{P})$, we shall consider the space $\mathscr{L}_{p}$ $(p > 0)$ of all measurable functions $X$ (necessarily random variables) for which $\mathbb{E} \abs{X}^{p} < \infty$. Given an event $A$ we shall denote the indicator random variable of the event $A$ by $I_{A}$. Throughout, the functions $x \mapsto \log \max\{x, \mathrm{e} \}$ and $x \mapsto \log \log \max\big\{x, \mathrm{e}^{\mathrm{e}} \big\}$ will be indicated by $\Log \, x$ and $\LLog \, x$, respectively; further, $g_{s,t}(x)$, $t > s \geqslant 0$ shall stands for the function $x \mapsto \max[\min(x - s,t),0]$. As usual and to make the computations be simpler looking we shall use the letter $C$ to denote any positive constant that can be explicitly computed, which is not necessarily the same on each appearance; the symbol $C(r)$ ($C(p,r)$ or $C(p,r,s)$) has the same meaning with the additional information that it depends on $r$ ($p,r$ or $p,r,s$). All over, $\lfloor x \rfloor$ and $\lceil x \rceil$ will be used to represent the largest integer not greater than $x$ and the smallest integer not less than $x$, respectively.

\section{Mainstream}

\indent

Let $\{X_{n}, \, n \geqslant 1 \}$ be a sequence of nonnegative random variables and suppose that, for any (constants) $t,s$ verifying $t > s \geqslant 0$, the (truncated) random sequence $\{g_{s,t}(X_{n}), \, n \geqslant 1 \}$ satisfies a moment inequality, namely there is a nondecreasing sequence of positive numbers $\{\lambda_{n} \}$ such that for some $r > 1$,
\begin{equation}\label{eq:2.1}
\mathbb{E} \abs{\sum_{k=\eta  + 1}^{\eta + n} \sum_{j=\xi_{k-1}+1}^{\xi_{k}} \big[g_{s,t}(X_{j}) - \mathbb{E} \,  g_{s,t}(X_{j}) \big]}^{r} \leqslant \lambda_{n}^{r} \sum_{k=\eta  + 1}^{\eta + n} \mathbb{E} \abs{\sum_{j=\xi_{k-1}+1}^{\xi_{k}}  \big[g_{s,t}(X_{j}) - \mathbb{E} \,  g_{s,t}(X_{j}) \big]}^{r}
\end{equation}
for all $\eta \geqslant 0$, $n \geqslant 1$ and every increasing sequence $\{\xi_{k} \}$ of nonnegative integers. It is well-known that if \eqref{eq:2.1} holds then
\begin{equation}\label{eq:2.2}
\mathbb{E} \left[\max_{1 \leqslant j \leqslant n} \abs{\sum_{k=\eta + 1}^{\eta + j} \sum_{i=\xi_{k-1}+1}^{\xi_{k}} \big[g_{s,t}(X_{i}) - \mathbb{E} \,  g_{s,t}(X_{i}) \big]} \right]^{r} \leqslant \Lambda_{n}^{r} \sum_{k=\eta + 1}^{\eta + n} \mathbb{E} \abs{\sum_{i=\xi_{k-1}+1}^{\xi_{k}}  \big[g_{s,t}(X_{i}) - \mathbb{E} \,  g_{s,t}(X_{i}) \big]}^{r}
\end{equation}
where $\Lambda_{1} := \lambda_{1}$ and $\Lambda_{n} := \lambda_{\lfloor(n + 2)/2 \rfloor} + \Lambda_{\lfloor(n + 2)/2 \rfloor - 1}$, $n \geqslant 2$ (see Theorem 4 of \cite{Moricz76}).

By employing both moment inequality \eqref{eq:2.1} and monotonicity of sums of nonnegative random variables, the following main statement provides us sufficient conditions for a sequence of random variables $\{X_{n}, \, n \geqslant 1 \}$ obeys the strong law of large numbers with respect to the normalising constants $\{b_{n} \}$. We emphasize that it is not admitted that the $X_{n}$'s have any particular dependence structure; the only restrictions on the dependence will be those imposed by assuming \eqref{eq:2.1}.

\begin{theorem}\label{thr:1}
Let $\{X_{n}, \, n \geqslant 1 \}$ be a sequence of nonnegative random variables satisfying $\mathbb{E} \, X_{n} < \infty$ for all $n \geqslant 1$ and verifying \eqref{eq:2.1} for some $r > 1$ and a nondecreasing sequence $\{\lambda_{n} \}$ of positive numbers. If $\{a_{n} \}$ is a sequence of nonnegative constants, $\{b_{n} \}$, $\{c_{n} \}$, $\{d_{n} \}$ are sequences of positive constants such that $c_{n} \leqslant d_{n}$ for all $n$, $b_{n}$ is nondecreasing unbounded, and there are increasing unbounded sequences $\{m_{k} \}$, $\{\ell_{k} \}$ of positive integers such that \\[-5pt]

\noindent \textnormal{(a)} ${\displaystyle \limsup_{k \rightarrow \infty}} \,$ $b_{m_{k+1}}/b_{m_{k}} < \infty$, ${\displaystyle \liminf_{k \rightarrow \infty}} \, \sum_{j=m_{k}}^{m_{k+1} - 1} a_{j} > 0$, \\[-5pt]

\noindent \textnormal{(b)} $\sum_{j=\ell_{k}+1}^{\ell_{k+1}} \mathbb{E} \, X_{j}I_{\left\{X_{j} > c_{j} \right\}} = o(b_{\ell_{k}})$ as $k \rightarrow \infty$, \\[-5pt]

\noindent \textnormal{(c)} $\sum_{n=1}^{\infty} \sum_{k=1}^{n} (a_{n} \Lambda_{n}^{r}/b_{n}^{r}) \mathbb{E} \, X_{k}^{r} I_{\{X_{k} \leqslant c_{k} \}} < \infty$, \\[-5pt]

\noindent \textnormal{(d)} $\sum_{n=1}^{\infty} \sum_{k=1}^{n} (a_{n} \Lambda_{n}^{r} c_{k}^{r}/b_{n}^{r}) \mathbb{P} \{X_{k} > c_{k} \} < \infty$, \\[-5pt]

\noindent \textnormal{(e)} $\sum_{n=1}^{\infty} \sum_{k=1}^{n} (a_{n}/b_{n}) \mathbb{E} \, X_{k} I_{\{X_{k} > d_{k} \}} < \infty$, \\[-5pt]

\noindent \textnormal{(f)} $\sum_{k=1}^{\infty} \sum_{j=\ell_{k}+1}^{\ell_{k+1}} \sum_{i=1}^{\ell_{k+1}} (a_{j} \Lambda_{k+1}^{r} \lambda_{\ell_{k} - \ell_{k-1}}^{r}/b_{\ell_{k}}^{r}) \mathbb{E} \, X_{i}^{r} I_{\{c_{i} < X_{i} \leqslant d_{i} \}} < \infty$, \\[-5pt]

\noindent \textnormal{(g)} $\sum_{k=1}^{\infty} \sum_{j=\ell_{k}+1}^{\ell_{k+1}} \sum_{i=1}^{\ell_{k+1}} (a_{j} \Lambda_{k+1}^{r} \lambda_{\ell_{k} - \ell_{k-1}}^{r} c_{i}^{r}/b_{\ell_{k}}^{r}) \mathbb{P} \{X_{i} > c_{i} \} < \infty$, \\[-5pt]

\noindent \textnormal{(h)} $\sum_{k=1}^{\infty} \sum_{j=\ell_{k}+1}^{\ell_{k+1}} \sum_{i=1}^{\ell_{k+1}} (a_{j} \Lambda_{k+1}^{r} \lambda_{\ell_{k} - \ell_{k-1}}^{r} d_{i}^{r}/b_{\ell_{k}}^{r}) \mathbb{P} \{X_{i} > d_{i} \} < \infty$, \\[-5pt]

\noindent then
\begin{equation*}
\frac{1}{b_{n}} \sum_{k = 1}^{n} (X_{k} - \mathbb{E} \, X_{k}) \overset{\textnormal{a.s.}}{\longrightarrow} 0.
\end{equation*}
\end{theorem}

Recall that a random sequence $\{X_{n}, \, n \geqslant 1 \}$ is \emph{stochastically dominated} by a random variable $X$ if there exists a constant $C>0$ such that
\begin{equation*}
\sup_{n \geqslant 1} \mathbb{P} \left\{\abs{X_{n}} > t \right\} \leqslant C \, \mathbb{P} \left\{\abs{X} > t  \right\},
\end{equation*}
for each $t > 0$. Additionally, a sequence $\{X_{n}, \, n \geqslant 1 \}$ of random variables is said to be \emph{pairwise negatively quadrant dependent} (or, for short, pairwise NQD) if
\begin{equation*}
\mathbb{P} \left\{X_{k} \leqslant x_{k}, X_{j} \leqslant x_{j}  \right\} - \mathbb{P} \left\{X_{k} \leqslant x_{k} \right\} \mathbb{P} \left\{X_{j} \leqslant x_{j}  \right\} \leqslant 0
\end{equation*}
for all reals $x_{k}, x_{j}$ and all positive integers $k,j$ such that $k \neq j$.

In the next result, we shall present an application of the previous main theorem by establishing a strong law of large numbers for sequences of pairwise negatively quadrant dependent random variables under sharp normalising constants.

\begin{corollary}\label{cor:1}
If $1 < p < 2$ and $\{X_{n}, \, n \geqslant 1 \}$ is a sequence of pairwise NQD random variables stochastically dominated by a random variable $X \in \mathscr{L}_{p}$, then
\begin{equation*}
\frac{1}{n^{1/p} (\LLog \, n)^{2(p - 1)/p}} \sum_{k=1}^{n} (X_{k} - \mathbb{E} \, X_{k}) \overset{\textnormal{a.s.}}{\longrightarrow} 0.
\end{equation*}
\end{corollary}

\begin{remark}
It is worthy to note that Corollary~\ref{cor:1} improves Corollary 2 of \cite{Lita18}.
\end{remark}

\begin{remark}
Let us observe that Marcinkiewicz-Zygmund strong law of large numbers for $p=1$ was already extended to sequences of pairwise NQD random variables, i.e. it was proved by Matu{\l}a in \cite{Matula92} that, if $\{X_{n}, \, n \geqslant 1 \}$ is a sequence of pairwise NQD and identically distributed random variables then, $(\sum_{k=1}^{n} X_{k} - nc)/n \overset{\textnormal{a.s.}}{\longrightarrow} 0$ for some finite constant $c$ if and only if $\mathbb{E} \lvert X_{1} \rvert < \infty$, and if so, $c = \mathbb{E} \, X_{1}$ (see Theorem 1 of \cite{Matula92}). 
\end{remark}

\begin{remark}
In \cite{Li06}, strong laws of large numbers were stated for sequences of pairwise negatively dependent random variables with normalizing constants $b_{n}$ satisfying $b_{n} = n \uparrow$. We stress out that this condition prevents us from choosing $b_{n} = n^{1/p} (\LLog \, n)^{2(p - 1)/p}$, $1 < p < 2$.
\end{remark}

\section{Lemmata and proofs}

\indent

We begin this section by revisiting Lemma 2.5 of \cite{Sung14}.

\begin{lemma}\label{lem:1}
Let $\{X_{n}, \, n \geqslant 1 \}$ be a sequence of nonnegative random variables, $\{b_{n} \}$ a nondecreasing unbounded sequence of positive constants, and $\{a_{n} \}$ a sequence of nonnegative constants. If there are increasing unbounded sequences $\{m_{k} \}$ and $\{\ell_{k} \}$ of positive integers such that \\[-5pt]

\noindent \textnormal{(i)} ${\displaystyle \limsup_{k \rightarrow \infty}} \, b_{m_{k+1}}/b_{m_{k}} < \infty$ and ${\displaystyle \liminf_{k \rightarrow \infty}} \, \sum_{j=m_{k}}^{m_{k+1}-1} a_{j} > 0$, \\[-2pt]

\noindent \textnormal{(ii)} $\sum_{j=\ell_{k}+1}^{\ell_{k+1}} \mathbb{E} \, X_{j} = o(b_{\ell_{k}})$ as $k \rightarrow \infty$, \\[-2pt]

\noindent \textnormal{(iii)} $\sum_{k=1}^{\infty} \mathbb{P} \Big\{{\displaystyle \max_{1 \leqslant n \leqslant k+1}} \big\lvert\sum_{i=1}^{\ell_{n}}(X_{i} - \mathbb{E} \, X_{i}) \big\rvert > \varepsilon b_{\ell_{k}} \Big\} \sum_{j=\ell_{k}+1}^{\ell_{k+1}} a_{j} < \infty$ for all $\varepsilon > 0$, \\[-2pt]

\noindent then $\sum_{k=1}^{n}(X_{k} - \mathbb{E} \, X_{k})/b_{n} \overset{\textnormal{a.s.}}{\longrightarrow} 0$.
\end{lemma}

\begin{proof}
According to Theorem 2.2 of \cite{Hu16}, it suffices to prove
\begin{equation}\label{eq:3.1}
\sum_{n=1}^{\infty} a_{n} \mathbb{P} \left\{\max_{1 \leqslant j \leqslant n} \abs{\sum_{i=1}^{j}(X_{i} - \mathbb{E} \, X_{i})} > \varepsilon b_{n} \right\} < \infty
\end{equation}
for all $\varepsilon > 0$. Setting $\ell_{0} := 0 =: b_{0}$, we have
\begin{align*}
&\sum_{n=1}^{\infty} a_{n} \mathbb{P} \left\{\max_{1 \leqslant j \leqslant n} \abs{\sum_{i=1}^{j}(X_{i} - \mathbb{E} \, X_{i})} > \varepsilon b_{n} \right\} \\
&\qquad \leqslant \sum_{k=0}^{\infty} \sum_{n=\ell_{k}+1}^{\ell_{k+1}} a_{n} \mathbb{P} \left\{\max_{1 \leqslant j \leqslant n} \abs{\sum_{i=1}^{j}(X_{i} - \mathbb{E} \, X_{i})} > \varepsilon b_{n} \right\} \\
&\qquad \leqslant \sum_{k=0}^{\infty} \mathbb{P} \left\{\max_{1 \leqslant j \leqslant \ell_{k+1}} \abs{\sum_{i=1}^{j}(X_{i} - \mathbb{E} \, X_{i})} > \varepsilon b_{\ell_{k}} \right\} \sum_{n=\ell_{k}+1}^{\ell_{k+1}} a_{n}.
\end{align*}
For any $\ell_{k} < j \leqslant \ell_{k+1}$ we get
\begin{align*}
\sum_{i=1}^{j} (X_{i} - \mathbb{E} \, X_{i}) &\leqslant \sum_{i=1}^{\ell_{k+1}} X_{i} - \sum_{i=1}^{\ell_{k}} \mathbb{E} \, X_{i} \\
&= \sum_{i=1}^{\ell_{k+1}} (X_{i} - \mathbb{E} \, X_{i}) + \sum_{i = \ell_{k}+1}^{\ell_{k+1}} \mathbb{E} \, X_{i} \\
&\leqslant \abs{\sum_{i=1}^{\ell_{k+1}} (X_{i} - \mathbb{E} \, X_{i})} + \abs{\sum_{i=1}^{\ell_{k}} (X_{i} - \mathbb{E} \, X_{i})} + \sum_{i = \ell_{k}+1}^{\ell_{k+1}} \mathbb{E} \, X_{i}
\end{align*}
and
\begin{align*}
\sum_{i=1}^{j} (X_{i} - \mathbb{E} \, X_{i}) &\geqslant \sum_{i=1}^{\ell_{k}} X_{i} - \sum_{i=1}^{\ell_{k+1}} \mathbb{E} \, X_{i} \\
&= \sum_{i=1}^{\ell_{k}} (X_{i} - \mathbb{E} \, X_{i}) - \sum_{i = \ell_{k}+1}^{\ell_{k+1}} \mathbb{E} \, X_{i} \\
&\geqslant - \abs{\sum_{i=1}^{\ell_{k+1}} (X_{i} - \mathbb{E} \, X_{i})} - \abs{\sum_{i=1}^{\ell_{k}} (X_{i} - \mathbb{E} \, X_{i})} - \sum_{i = \ell_{k}+1}^{\ell_{k+1}} \mathbb{E} \, X_{i}
\end{align*}
which yields
\begin{equation*}
\max_{\ell_{k} < j \leqslant \ell_{k+1}} \abs{\sum_{i=1}^{j} (X_{i} - \mathbb{E} \, X_{i})} \leqslant \abs{\sum_{i=1}^{\ell_{k+1}} (X_{i} - \mathbb{E} \, X_{i})} + \abs{\sum_{i=1}^{\ell_{k}} (X_{i} - \mathbb{E} \, X_{i})} + \sum_{i = \ell_{k}+1}^{\ell_{k+1}} \mathbb{E} \, X_{i}.
\end{equation*}
Hence,
\begin{align*}
\max_{0 < j \leqslant \ell_{n+1}} \abs{\sum_{i=1}^{j} (X_{i} - \mathbb{E} \, X_{i})} &= \max_{0 \leqslant k \leqslant n} \left[\max_{\ell_{k} < j \leqslant \ell_{k+1}} \abs{\sum_{i=1}^{j} (X_{i} - \mathbb{E} \, X_{i})} \right] \\
&\leqslant 2 \max_{1 \leqslant k \leqslant n+1} \abs{\sum_{i=1}^{\ell_{k}} (X_{i} - \mathbb{E} \, X_{i})} + \max_{0 \leqslant k \leqslant n} \sum_{i = \ell_{k}+1}^{\ell_{k+1}} \mathbb{E} \, X_{i}.
\end{align*}
From Lemma 2.3 of \cite{Sung14}, we have
\begin{equation*}
\frac{1}{b_{\ell_{n}}} \max_{0 \leqslant k \leqslant n} \sum_{i=\ell_{k}+1}^{\ell_{k+1}} \mathbb{E} \, X_{i} = o(1), \quad n \rightarrow \infty
\end{equation*}
so that, for each $\varepsilon > 0$ fix, there exists a positive integer $n_{0}$ such that
\begin{equation*}
\max_{1 \leqslant j \leqslant \ell_{n+1}} \abs{\sum_{i=1}^{j} (X_{i} - \mathbb{E} \, X_{i})} \leqslant 2 \max_{1 \leqslant k \leqslant n+1} \abs{\sum_{i=1}^{\ell_{k}} (X_{i} - \mathbb{E} \, X_{i})} + \frac{\varepsilon b_{\ell_{n}}}{2}
\end{equation*}
for all $n \geqslant n_{0}$. Thus,
\begin{align*}
&\sum_{n=n_{0}}^{\infty} \mathbb{P} \left\{\max_{1 \leqslant j \leqslant \ell_{n+1}} \abs{\sum_{i=1}^{j}(X_{i} - \mathbb{E} \, X_{i})} > \varepsilon b_{\ell_{n}} \right\} \sum_{k=\ell_{n}+1}^{\ell_{n+1}} a_{k} \\
&\qquad \leqslant \sum_{n=n_{0}}^{\infty} \mathbb{P} \left\{\max_{1 \leqslant j \leqslant n+1} \abs{\sum_{i=1}^{\ell_{j}}(X_{i} - \mathbb{E} \, X_{i})} > \frac{\varepsilon b_{\ell_{n}}}{4} \right\} \sum_{k=\ell_{n}+1}^{\ell_{n+1}} a_{k} < \infty
\end{align*}
by assumption (iii) and \eqref{eq:3.1} holds. The thesis is established.
\end{proof}

The following lemmas will play a central role in the proof of Corollary~\ref{cor:1}.

\begin{lemma}\label{lem:2}
Let $\{X_{n}, \, n \geqslant 1 \}$ be a sequence of random variables stochastically dominated by a random variable $X \in \mathscr{L}_{p}$ for some $0 < p < 2$. If $\varphi$ is a real-valued function defined on whole real line such that $\lvert \varphi(x) \rvert \leqslant \lvert x \rvert$ for all $x$ then, for every $r > p$,
\begin{equation}\label{eq:3.2}
\sum_{n=1}^{\infty} \sum_{k=1}^{n} \frac{\Log^{r} n}{n^{r/p + 1}} \mathbb{E} \lvert \varphi(X_{k}) \rvert^{r} I_{\left\{\lvert \varphi(X_{k}) \rvert \leqslant \frac{k^{1/p}}{\Log^{r/(r - p)} k} \right\}} < \infty
\end{equation}
and
\begin{equation}\label{eq:3.3}
\sum_{n=1}^{\infty} \sum_{k=1}^{n} \frac{k^{r/p} \Log^{r} n}{n^{r/p + 1} \Log^{r^{2}/(r - p)}k} \mathbb{P} \left\{\lvert \varphi(X_{k}) \rvert > \frac{k^{1/p}}{\Log^{r/(r - p)}k} \right\} < \infty.
\end{equation}
Furthermore, if $p > 1$ then
\begin{equation}\label{eq:3.4}
\sum_{n=1}^{\infty} \sum_{k=1}^{n} \frac{1}{n^{1/p + 1} (\LLog \, n)^{r(p - 1)/[p(r - 1)]}} \mathbb{E} \lvert \varphi(X_{k}) \rvert I_{\left\{\lvert \varphi(X_{k}) \rvert > \frac{k^{1/p}}{(\LLog \, k)^{r/[p(r - 1)]}} \right\}} < \infty.
\end{equation}
\end{lemma}

\begin{proof}
Recall that Stolz--Ces\`{a}ro theorem ensures
\begin{gather}
\sum_{k=1}^{n} \frac{\Log^{\alpha} k}{k^{\beta}} \sim \frac{n^{1 - \beta}\Log^{\alpha} n}{1 - \beta}, \quad n \rightarrow \infty \label{eq:3.5} \\
\sum_{k=n}^{\infty} \frac{\Log^{\alpha} k}{k^{\delta}} \sim  \frac{n^{1 - \delta} \Log^{\alpha} n}{\delta - 1}, \quad n \rightarrow \infty \label{eq:3.6} \\
\sum_{k=1}^{n} \frac{\LLog^{\alpha} k}{k^{\beta}} \sim \frac{n^{1 - \beta} \LLog^{\alpha} n}{1 - \beta}, \quad n \rightarrow \infty \label{eq:3.7} \\
\sum_{k=n}^{\infty} \frac{\LLog^{\alpha} k}{k^{\delta}} \sim  \frac{n^{1 - \delta} \LLog^{\alpha} n}{\delta - 1}, \quad n \rightarrow \infty \label{eq:3.8}
\end{gather}
for all reals $\alpha, \beta, \delta$ such that $\beta < 1$, $\delta > 1$. Suppose
\begin{equation*}
A_{k} = \left\{\frac{(k - 1)^{1/p}}{\Log^{r/(r - p)}(k - 1)} < \lvert X \rvert \leqslant \frac{k^{1/p}}{\Log^{r/(r - p)}k} \right\}, \quad k \geqslant 1.
\end{equation*}
Hence, from \eqref{eq:3.5}, \eqref{eq:3.6} and Lemma 3 of \cite{Lita15} we have
\begin{align*}
&\sum_{n=1}^{\infty} \sum_{k=1}^{n} \frac{\Log^{r} n}{n^{r/p + 1}} \mathbb{E} \lvert \varphi(X_{k}) \rvert^{r} I_{\left\{\lvert \varphi(X_{k}) \rvert \leqslant \frac{k^{1/p}}{\Log^{r/(r - p)} k} \right\}} \\
&\quad = \sum_{k=1}^{\infty} \sum_{n=k}^{\infty} \frac{\Log^{r} n}{n^{r/p + 1}} \mathbb{E} \lvert \varphi(X_{k}) \rvert^{r} I_{\left\{\lvert \varphi(X_{k}) \rvert \leqslant \frac{k^{1/p}}{\Log^{r/(r - p)} k} \right\}} \\
&\quad \leqslant C(p,r) \sum_{k=1}^{\infty} \frac{\Log^{r} k}{k^{r/p}} \mathbb{E} \lvert \varphi(X_{k}) \rvert^{r} I_{\left\{\lvert \varphi(X_{k}) \rvert^{r} \leqslant \frac{k^{r/p}}{\Log^{r^{2}/(r - p)} k} \right\}} \\
&\quad \leqslant C(p,r) \sum_{k=1}^{\infty} \frac{\Log^{r} k}{k^{r/p}} \int_{0}^{k^{r/p}/\Log^{r^{2}/(r - p)} k} \mathbb{P} \left\{\lvert \varphi(X_{k}) \rvert^{r} > u \right\} \, \mathrm{d}u \\
&\quad \leqslant C(p,r) \sum_{k=1}^{\infty} \frac{\Log^{r} k}{k^{r/p}} \int_{0}^{k^{r/p}/\Log^{r^{2}/(r - p)} k} \mathbb{P} \left\{\lvert X_{k} \rvert^{r} > u \right\} \, \mathrm{d}u \\
&\quad \leqslant C(p,r) \sum_{k=1}^{\infty} \frac{\Log^{r} k}{k^{r/p}} \int_{0}^{k^{r/p}/\Log^{r^{2}/(r - p)} k} \mathbb{P} \left\{\lvert X \rvert^{r} > u \right\} \, \mathrm{d}u \\
&\quad \leqslant C(p,r) \sum_{k=1}^{\infty} \frac{\Log^{r} k}{k^{r/p}} \mathbb{E} \lvert X \rvert^{r} I_{\left\{\lvert X \rvert^{r} \leqslant \frac{k^{r/p}}{\Log^{r^{2}/(r - p)} k} \right\}} \\
&\qquad + C(p,r) \sum_{k=1}^{\infty} \frac{1}{\Log^{rp/(r - p)}k} \mathbb{P} \left\{\lvert X \rvert > \frac{k^{1/p}}{\Log^{r/(r - p)} k} \right\} \\
&\quad = C(p,r) \sum_{k=1}^{\infty} \sum_{j=1}^{k} \frac{\Log^{r} k}{k^{r/p}} \mathbb{E} \lvert X \rvert^{r} I_{A_{j}} + C(p,r) \sum_{k=1}^{\infty} \frac{1}{\Log^{rp/(r - p)}k} \mathbb{P} \left\{\lvert X \rvert^{p} > \frac{k}{\Log^{rp/(r - p)} k} \right\} \\
&\quad \leqslant C(p,r) \sum_{j=1}^{\infty} \sum_{k=j}^{\infty} \frac{\Log^{r} k}{k^{r/p}} \mathbb{E} \lvert X \rvert^{r} I_{A_{j}} + C(p,r) \, \mathbb{E} \, \lvert X \rvert^{p} \\
&\quad \leqslant C(p,r) \sum_{j=1}^{\infty} \frac{\Log^{r}j}{j^{r/p - 1}} \mathbb{E} \lvert X \rvert^{r} I_{A_{j}} + C(p,r) \, \mathbb{E} \, \lvert X \rvert^{p} \\
&\quad \leqslant C(p,r) \sum_{j=1}^{\infty} \frac{\Log^{r}j}{j^{r/p - 1}} \cdot \frac{j^{(r - p)/p}}{\Log^{r} j} \mathbb{E} \, \lvert X \rvert^{p} I_{A_{j}} + C \, \mathbb{E} \, \lvert X \rvert^{p} \\
&\quad \leqslant C(p,r) \, \mathbb{E} \, \lvert X \rvert^{p} < \infty
\end{align*}
and
\begin{align*}
&\sum_{n=1}^{\infty} \sum_{k=1}^{n} \frac{k^{r/p} \Log^{r} n}{n^{r/p + 1} \Log^{r^{2}/(r - p)}k} \mathbb{P} \left\{\lvert \varphi(X_{k}) \rvert > \frac{k^{1/p}}{\Log^{r/(r - p)}k} \right\} \\
&\qquad = \sum_{k=1}^{\infty} \sum_{n=k}^{\infty} \frac{k^{r/p} \Log^{r} n}{n^{r/p + 1} \Log^{r^{2}/(r - p)}k} \mathbb{P} \left\{\lvert \varphi(X_{k}) \rvert > \frac{k^{1/p}}{\Log^{r/(r - p)}k} \right\} \\
&\qquad \leqslant C(p,r) \sum_{k=1}^{\infty} \frac{\Log^{r}k}{\Log^{r^{2}/(r - p)}k} \mathbb{P} \left\{\lvert X_{k} \rvert > \frac{k^{1/p}}{\Log^{r/(r - p)}k} \right\} \\
&\qquad \leqslant C(p,r) \sum_{k=1}^{\infty} \frac{1}{\Log^{rp/(r - p)}k} \mathbb{P} \left\{\lvert X \rvert > \frac{k^{1/p}}{\Log^{r/(r - p)}k} \right\} \\
&\qquad \leqslant C(p,r) \sum_{k=1}^{\infty} \frac{1}{\Log^{rp/(r - p)}k} \mathbb{P} \left\{\lvert X \rvert^{p} > \frac{k}{\Log^{rp/(r - p)}k} \right\} \\
&\qquad \leqslant C(p,r) \, \mathbb{E} \, \lvert X \rvert^{p} < \infty.
\end{align*}
Moreover, putting
\begin{equation}\label{eq:3.9}
B_{k} = \left\{\frac{(k - 1)^{1/p}}{[\LLog (k - 1)]^{r/[p(r - 1)]}} < \lvert X \rvert \leqslant \frac{k^{1/p}}{(\LLog \, k)^{r/[p(r - 1)]}} \right\}, \quad k \geqslant 1
\end{equation}
Lemma 1 of \cite{Lita15} implies
\begin{align*}
&\sum_{n=1}^{\infty} \sum_{k=1}^{n} \frac{1}{n^{1/p + 1} (\LLog \, n)^{r(p - 1)/[p(r - 1)]}} \mathbb{E} \lvert \varphi(X_{k}) \rvert I_{\left\{\lvert \varphi(X_{k}) \rvert > \frac{k^{1/p}}{(\LLog \, k)^{r/[p(r - 1)]}} \right\}} \\
&\quad = \sum_{k=1}^{\infty} \sum_{n=k}^{\infty} \frac{1}{n^{1/p + 1} (\LLog \, n)^{r(p - 1)/[p(r - 1)]}} \mathbb{E} \lvert \varphi(X_{k}) \rvert I_{\left\{\lvert \varphi(X_{k}) \rvert > \frac{k^{1/p}}{(\LLog \, k)^{r/[p(r - 1)]}} \right\}} \\
&\quad \leqslant C(p,r) \sum_{k=1}^{\infty} \frac{1}{k^{1/p} (\LLog \, k)^{r(p - 1)/[p(r - 1)]}} \mathbb{E} \lvert \varphi(X_{k}) \rvert I_{\left\{\lvert \varphi(X_{k}) \rvert > \frac{k^{1/p}}{(\LLog \, k)^{r/[p(r - 1)]}} \right\}} \\
&\quad \leqslant C(p,r) \sum_{k=1}^{\infty} \frac{1}{k^{1/p} (\LLog \, k)^{r(p - 1)/[p(r - 1)]}} \mathbb{E} \, \lvert X_{k} \rvert I_{\left\{\lvert X_{k} \rvert > \frac{k^{1/p}}{(\LLog \, k)^{r/[p(r - 1)]}} \right\}} \\
&\quad \leqslant C(p,r) \sum_{k=1}^{\infty} \frac{1}{k^{1/p} (\LLog \, k)^{r(p - 1)/[p(r - 1)]}} \mathbb{E} \, \lvert X \rvert I_{\left\{\lvert X \rvert > \frac{k^{1/p}}{(\LLog \, k)^{r/[p(r - 1)]}} \right\}} \\
&\quad \leqslant C(p,r) \sum_{k=1}^{\infty} \sum_{j=k+1}^{\infty} \frac{1}{k^{1/p} (\LLog \, k)^{r(p - 1)/[p(r - 1)]}} \mathbb{E} \, \lvert X \rvert I_{B_{j}} \\
&\quad = C(p,r) \sum_{j=2}^{\infty} \sum_{k=1}^{j-1} \frac{1}{k^{1/p} (\LLog \, k)^{}} \mathbb{E} \, \lvert X \rvert I_{B_{j}} \\
&\quad \leqslant C(p,r) \sum_{j=2}^{\infty} \frac{(j - 1)^{1 - 1/p}}{[\LLog \, (j - 1)]^{r(p - 1)/[p(r - 1)]}} \mathbb{E} \, \lvert X \rvert I_{B_{j}} \\
&\quad \leqslant C(p,r) \sum_{j=2}^{\infty} \frac{(j - 1)^{1 - 1/p}}{[\LLog \, (j - 1)]^{r(p - 1)/[p(r - 1)]}} \cdot \frac{(j - 1)^{(1 - p)/p}}{[\LLog \, (j - 1)]^{r(1 - p)/[p(r - 1)]}} \mathbb{E} \, \lvert X \rvert^{p} I_{B_{j}} \\
&\quad \leqslant C(p,r) \, \mathbb{E} \, \lvert X \rvert^{p} < \infty
\end{align*}
by using \eqref{eq:3.7} and \eqref{eq:3.8}. The proof is complete.
\end{proof}

\begin{lemma}\label{lem:3}
Let $a,b > 0$ and $r$ be a real number. Then, for any $x \geqslant 0$,
\begin{equation*}
\int_{x}^{\infty} \frac{u^{a - 1} \Log^{r} u}{\mathrm{e}^{b u^{a}}} \, \mathrm{d}u < \infty.
\end{equation*}
Furthermore,
\begin{equation}\label{eq:3.10}
\int_{x}^{\infty} \frac{u^{a - 1} \Log^{r} u}{\mathrm{e}^{b u^{a}}} \, \mathrm{d}u = O \left[(1 + \Log^{r} x) \mathrm{e}^{-b x^{a}} \right], \quad x \rightarrow \infty.
\end{equation}
\end{lemma}

\begin{proof}
Supposing $x \geqslant \mathrm{e}$, elementary integration by parts yields
\begin{align*}
\int_{x}^{\infty} \frac{u^{a - 1} \Log^{r} u}{\mathrm{e}^{b u^{a}}} \, \mathrm{d}u &= \int_{x}^{\infty} \frac{u^{a - 1} \log^{r} u}{\mathrm{e}^{b u^{a}}} \, \mathrm{d}u \\
&= \frac{\log^{r} x}{a b \mathrm{e}^{b x^{a}}} + \frac{r}{a b} \int_{x}^{\infty} \frac{\log^{r - 1} u}{u \mathrm{e}^{b u^{a}}} \, \mathrm{d}u \\
&= \frac{\log^{r} x}{a b \mathrm{e}^{b x^{a}}} + \frac{r}{a b} \int_{x}^{\infty} \frac{\log^{r - 1} u}{u^{a}} \cdot \frac{u^{a - 1}}{\mathrm{e}^{b u^{a}}} \, \mathrm{d}u \\
&\leqslant \frac{\log^{r} x}{a b \mathrm{e}^{b x^{a}}} + \frac{r C}{ab} \int_{x}^{\infty} \frac{u^{a - 1}}{\mathrm{e}^{b u^{a}}} \, \mathrm{d}u \\
&= \frac{\log^{r} x}{a b \mathrm{e}^{b x^{a}}} + \frac{r C}{(a b)^{2} \mathrm{e}^{b x^{a}}}
\end{align*}
where $C$ is a positive constant depending only on $a,r$ and \eqref{eq:3.10} holds. For $x < \mathrm{e}$, we still have
\begin{align*}
\int_{x}^{\infty} \frac{u^{a - 1} \Log^{r} u}{\mathrm{e}^{b u^{a}}} \, \mathrm{d}u &= \int_{x}^{\mathrm{e}} \frac{u^{a - 1}}{\mathrm{e}^{b u^{a}}} \, \mathrm{d}u + \int_{\mathrm{e}}^{\infty} \frac{u^{a - 1} \Log^{r} u}{\mathrm{e}^{b u^{a}}} \, \mathrm{d}u \\
&\leqslant \left(\frac{1}{a b \mathrm{e}^{b x^{a}}} - \frac{1}{a b \mathrm{e}^{b \mathrm{e}^{a}}} \right) + \frac{\log^{r} \mathrm{e}}{a b \mathrm{e}^{b \mathrm{e}^{a}}} + \frac{r C}{(a b)^{2} \mathrm{e}^{b \mathrm{e}^{a}}} \\
&= \frac{1}{a b \mathrm{e}^{b x^{a}}} + \frac{r C}{(a b)^{2} \mathrm{e}^{b \mathrm{e}^{a}}}
\end{align*}
and the conclusion follows.
\end{proof}

\begin{lemma}\label{lem:4}
Let $\{X_{n}, \, n \geqslant 1 \}$ be a sequence of random variables stochastically dominated by a random variable $X \in \mathscr{L}_{p}$ for some $1 < p < 2$ and $\ell_{k} := \big\lfloor\mathrm{e}^{{k}^{s}} \big\rfloor$, $0 < s < 1$. If $\varphi$ is a real-valued function defined on whole real line such that $\lvert \varphi(x) \rvert \leqslant \lvert x \rvert$ for all $x$ then for $r>p$,
\begin{gather}
\sum_{k=1}^{\infty} \sum_{n=1}^{\ell_{k+1}} \sum_{j=\ell_{k}+1}^{\ell_{k+1}} \frac{\Log^{r} k}{j \ell_{k}^{r/p} (\LLog \, \ell_{k})^{r^{2}(p - 1)/[p(r - 1)]}} \mathbb{E} \lvert \varphi(X_{n}) \rvert^{r} I_{\left\{\lvert \varphi(X_{n}) \rvert \leqslant \frac{n^{1/p}}{(\LLog \, n)^{r/[p(r - 1)]}} \right\}} < \infty, \label{eq:3.11} \\
\sum_{k=1}^{\infty} \sum_{n=1}^{\ell_{k+1}} \sum_{j=\ell_{k}+1}^{\ell_{k+1}} \frac{\Log^{r} k \, n^{r/p} \Log^{-r^{2}/(r - p)} n}{j \ell_{k}^{r/p} (\LLog \, \ell_{k})^{r^{2}(p - 1)/[p(r - 1)]}} \mathbb{P} \left\{\lvert \varphi(X_{n}) \rvert > \frac{n^{1/p}}{\Log^{r/(r - p)} n} \right\} < \infty \label{eq:3.12}
\end{gather}
and
\begin{equation}\label{eq:3.13}
\sum_{k=1}^{\infty} \sum_{n=1}^{\ell_{k+1}} \sum_{j=\ell_{k}+1}^{\ell_{k+1}} \frac{\Log^{r} k \, n^{r/p} (\LLog \, n)^{-r^{2}/[p(r - 1)]}}{j \ell_{k}^{r/p} (\LLog \, \ell_{k})^{r^{2}(p - 1)/[p(r - 1)]}} \mathbb{P} \left\{\lvert \varphi(X_{n}) \rvert > \frac{n^{1/p}}{(\LLog \, n)^{r/[p(r - 1)]}} \right\} < \infty.
\end{equation}
Furthermore, if $s \leqslant (r - p)/[p(r - 1)]$ then
\begin{equation}\label{eq:3.14}
\sum_{j=\ell_{k}+1}^{\ell_{k+1}} \frac{\mathbb{E} \lvert \varphi(X_{j}) \rvert I_{\left\{\lvert \varphi(X_{j}) \rvert > j^{1/p}/\Log^{r/(r - p)} j \right\}}}{\ell_{k}^{1/p} (\LLog \, \ell_{k})^{r(p - 1)/[p(r - 1)]}} = o(1), \quad n \rightarrow \infty.
\end{equation}
\end{lemma}

\begin{proof}
Noticing that
\begin{gather}
\ell_{k+1} \sim \ell_{k}, \quad k \rightarrow \infty \label{eq:3.15} \\
\frac{\ell_{k+1} - \ell_{k}}{\ell_{k}} \sim s k^{s-1}, \quad k \rightarrow \infty \label{eq:3.16} \\
\LLog \, \ell_{k+1} \sim s \Log \, k, \quad k \rightarrow \infty \label{eq:3.17}
\end{gather}
it follows
\begin{align*}
&\sum_{k=1}^{\infty} \sum_{n=1}^{\ell_{k+1}} \sum_{j=\ell_{k}+1}^{\ell_{k+1}} \frac{\Log^{r} k}{j \ell_{k}^{r/p} (\LLog \, \ell_{k})^{r^{2}(p - 1)/[p(r - 1)]}} \mathbb{E} \lvert \varphi(X_{n}) \rvert^{r} I_{\left\{\lvert \varphi(X_{n}) \rvert \leqslant \frac{n^{1/p}}{(\LLog \, n)^{r/[p(r - 1)]}} \right\}} \\
&\quad \leqslant \sum_{k=1}^{\infty} \sum_{n=1}^{\ell_{k+1}} \frac{(\ell_{k+1} - \ell_{k})\Log^{r} k}{\ell_{k}^{r/p + 1} (\LLog \, \ell_{k})^{r^{2}(p - 1)/[p(r - 1)]}} \mathbb{E} \lvert \varphi(X_{n}) \rvert^{r} I_{\left\{\lvert \varphi(X_{n}) \rvert \leqslant \frac{n^{1/p}}{(\LLog \, n)^{r/[p(r - 1)]}} \right\}} \\
&\quad \leqslant \sum_{k=1}^{\infty} \sum_{n=1}^{\ell_{k+1}} \frac{(\ell_{k+1} - \ell_{k}) \Log^{r} k}{\ell_{k}^{r/p + 1} (\LLog \, \ell_{k})^{r^{2}(p - 1)/[p(r - 1)]}} \int_{0}^{n^{r/p}/(\LLog \, n)^{r^{2}/[p(r - 1)]}} \mathbb{P} \left\{\lvert \varphi(X_{k}) \rvert^{r} > u \right\} \mathrm{d}u \\
&\quad \leqslant \sum_{k=1}^{\infty} \sum_{n=1}^{\ell_{k+1}} \frac{(\ell_{k+1} - \ell_{k}) \Log^{r} k}{\ell_{k}^{r/p + 1} (\LLog \, \ell_{k})^{r^{2}(p - 1)/[p(r - 1)]}} \int_{0}^{n^{r/p}/(\LLog \, n)^{r^{2}/[p(r - 1)]}} \mathbb{P} \left\{\lvert X_{k} \rvert^{r} > u \right\} \mathrm{d}u \\
&\quad \leqslant C \sum_{k=1}^{\infty} \sum_{n=1}^{\ell_{k+1}} \frac{(\ell_{k+1} - \ell_{k}) \Log^{r} k}{\ell_{k}^{r/p + 1} (\LLog \, \ell_{k})^{r^{2}(p - 1)/[p(r - 1)]}} \int_{0}^{n^{r/p}/(\LLog \, n)^{r^{2}/[p(r - 1)]}} \mathbb{P} \left\{\lvert X \rvert^{r} > u \right\} \mathrm{d}u \\
&\quad = C \sum_{k=1}^{\infty} \sum_{n=1}^{\ell_{k+1}} \frac{(\ell_{k+1} - \ell_{k})\Log^{r} k}{\ell_{k}^{r/p + 1} (\LLog \, \ell_{k})^{r^{2}(p - 1)/[p(r - 1)]}} \mathbb{E} \lvert X \rvert^{r} I_{\left\{\lvert X \rvert \leqslant \frac{n^{1/p}}{(\LLog \, n)^{r/[p(r - 1)]}} \right\}} \\
&\qquad + C \sum_{k=1}^{\infty} \sum_{n=1}^{\ell_{k+1}} \frac{(\ell_{k+1} - \ell_{k})\Log^{r} k \, n^{r/p}(\LLog \, n)^{-r^{2}/[p(r - 1)]}}{\ell_{k}^{r/p + 1} (\LLog \, \ell_{k})^{r^{2}(p - 1)/[p(r - 1)]}} \mathbb{P} \left\{\lvert X \rvert > \frac{n^{1/p}}{(\LLog \, n)^{r/[p(r - 1)]}} \right\} \\
& \quad \leqslant C(p,r,s) \sum_{k=1}^{\infty} \sum_{n=1}^{\ell_{k+1}} \frac{k^{s-1} \Log^{r(r - p)/[p(r - 1)]} k}{\ell_{k}^{r/p}} \mathbb{E} \lvert X \rvert^{r} I_{\left\{\lvert X \rvert \leqslant \frac{n^{1/p}}{(\LLog \, n)^{r/[p(r - 1)]}} \right\}} \\
& \qquad + C(p,r,s) \sum_{k=1}^{\infty} \sum_{n=1}^{\ell_{k+1}} \frac{k^{s-1} \Log^{r(r - p)/[p(r - 1)]} k \, n^{r/p}}{\ell_{k}^{r/p} (\LLog \, n)^{r^{2}/[p(r - 1)]}} \mathbb{P} \left\{\lvert X \rvert > \frac{n^{1/p}}{(\LLog \, n)^{r/[p(r - 1)]}} \right\} \\
&\quad \leqslant C(p,r,s) \sum_{k=1}^{\infty} \frac{\ell_{k+1} k^{s-1} \Log^{r(r - p)/[p(r - 1)]} k}{\ell_{k}^{r/p}} \mathbb{E} \lvert X \rvert^{r} I_{\left\{\lvert X \rvert \leqslant \frac{\ell_{k+1}^{1/p}}{(\LLog \, \ell_{k+1})^{r/[p(r - 1)]}} \right\}} \\
&\qquad +  C(p,r,s) \sum_{k=1}^{\infty} \sum_{n=1}^{\ell_{k+1}} \frac{k^{s-1} \Log^{r(r - p)/[p(r - 1)]} k \, n^{r/p}}{\mathrm{e}^{(r/p)k^{s}} (\LLog \, n)^{r^{2}/[p(r - 1)]}} \mathbb{P} \left\{\lvert X \rvert > \frac{n^{1/p}}{(\LLog \, n)^{r/[p(r - 1)]}} \right\} \\
&\quad \leqslant C(p,r,s) \sum_{k=1}^{\infty} \sum_{n=1}^{\ell_{k+1}} \frac{k^{s-1} \Log^{r(r - p)/[p(r - 1)]} k}{\mathrm{e}^{(r/p - 1)k^{s}}} \mathbb{E} \lvert X \rvert^{r} I_{B_{n}} \\
&\qquad + C(p,r,s) \sum_{n=1}^{\infty} \sum_{\left\{k \colon \ell_{k+1} \geqslant n \right\}} \frac{k^{s-1} \Log^{r(r - p)/[p(r - 1)]} k \, n^{r/p}}{\mathrm{e}^{(r/p)k^{s}} (\LLog \, n)^{r^{2}/[p(r - 1)]}} \mathbb{P} \left\{\lvert X \rvert > \frac{n^{1/p}}{(\LLog \, n)^{r/[p(r - 1)]}} \right\} \\
&\quad = C(p,r,s) \sum_{n=1}^{\infty} \sum_{\left\{k \colon \ell_{k+1} \geqslant n \right\}} \frac{k^{s-1} \Log^{r(r - p)/[p(r - 1)]} k}{\mathrm{e}^{(r/p - 1)k^{s}}} \mathbb{E} \lvert X \rvert^{r} I_{B_{n}} \\
&\qquad + C(p,r,s) \sum_{n=1}^{\infty} \sum_{\left\{k \colon \ell_{k+1} \geqslant n \right\}} \frac{k^{s-1} \Log^{r(r - p)/[p(r - 1)]} k \, n^{r/p}}{\mathrm{e}^{(r/p)k^{s}} (\LLog \, n)^{r^{2}/[p(r - 1)]}} \mathbb{P} \left\{\lvert X \rvert > \frac{n^{1/p}}{(\LLog \, n)^{r/[p(r - 1)]}} \right\}
\end{align*}
where the summation ${\displaystyle \sum_{\left\{k\colon \ell_{k+1} \geqslant n \right\}}}$ is taken over all positive integers $k$ such that $\ell_{k+1} \geqslant n$ and $B_{n}$ is given by \eqref{eq:3.9}. Since
\begin{align*}
\ell_{k+1} := \big\lfloor \mathrm{e}^{(k+1)^{s}} \big\rfloor \geqslant n \quad &\Longleftrightarrow \quad \mathrm{e}^{(k+1)^{s}} \geqslant \lfloor n \rfloor^{-} = \lceil n \rceil = n \\
&\Longleftrightarrow \quad k \geqslant \left(\log n \right)^{1/s} - 1
\end{align*}
(see Proposition 1 of \cite{Embrechts13}), Lemma~\ref{lem:3} implies, for $n$ large enough,
\begin{align*}
&\sum_{\left\{k \colon \ell_{k+1} \geqslant n \right\}} \frac{k^{s-1} \Log^{r(r - p)/[p(r - 1)]} k}{\mathrm{e}^{(r/p - 1)k^{s}}} = \sum_{k=\varphi_{s}(n)}^{\infty} \frac{k^{s-1} \Log^{r(r - p)/[p(r - 1)]} k}{\mathrm{e}^{(r/p - 1)k^{s}}} \\
&\qquad \leqslant \int_{\varphi_{s}(n) - 1}^{\infty} \frac{u^{s-1} \Log^{r(r - p)/[p(r - 1)]} u}{\mathrm{e}^{(r/p - 1)u^{s}}} \, \mathrm{d}u = O \left[\frac{(\LLog \, n)^{r(r - p)/[p(r - 1)]}}{n^{r/p - 1}} \right]
\end{align*}
and
\begin{equation}\label{eq:3.18}
\begin{split}
&\sum_{\left\{k \colon \ell_{k+1} \geqslant n \right\}} \frac{k^{s-1} \Log^{r(r - p)/[p(r - 1)]} k}{\mathrm{e}^{(r/p)k^{s}}} = \sum_{k=\varphi_{s}(n)}^{\infty} \frac{k^{s-1} \Log^{r(r - p)/[p(r - 1)]} k}{\mathrm{e}^{(r/p)k^{s}}} \\
&\qquad \leqslant \int_{\varphi_{s}(n) - 1}^{\infty} \frac{u^{s-1} \Log^{r(r - p)/[p(r - 1)]} u}{\mathrm{e}^{(r/p)u^{s}}} \, \mathrm{d}u = O \left[\frac{(\LLog \, n)^{r(r - p)/[p(r - 1)]}}{n^{r/p}} \right]
\end{split}
\end{equation}
with
\begin{equation*}
\varphi_{s}(n) := \max \big\{\big\lceil(\log n)^{1/s} \big\rceil - 1,1 \big\} \sim \Log^{1/s} n, \quad n \rightarrow \infty.
\end{equation*}
Thus,
\begin{align*}
&\sum_{n=1}^{\infty} \sum_{\left\{k \colon \ell_{k+1} \geqslant n \right\}} \frac{k^{s-1} \Log^{r(r - p)/[p(r - 1)]} k}{\mathrm{e}^{(r/p - 1)k^{s}}} \mathbb{E} \lvert X \rvert^{r} I_{B_{n}} \\
& \qquad \leqslant C(p,r,s) \sum_{n=1}^{\infty} \frac{(\LLog \, n)^{r(r - p)/[p(r - 1)]}}{n^{r/p - 1}} \mathbb{E} \lvert X \rvert^{r} I_{B_{n}} \\
& \qquad \leqslant C(p,r,s) \sum_{n=1}^{\infty} \frac{(\LLog \, n)^{r(r - p)/[p(r - 1)]}}{n^{r/p - 1}} \frac{n^{(r - p)/p}}{(\LLog \, n)^{r(r - p)/[p(r - 1)]}} \mathbb{E} \lvert X \rvert^{p} I_{B_{n}} \\
& \qquad = C(p,r,s) \mathbb{E} \lvert X \rvert^{p} < \infty
\end{align*}
and Lemma 3 of \cite{Lita15} entails
\begin{align*}
&\sum_{n=1}^{\infty} \sum_{\left\{k \colon \ell_{k+1} \geqslant n \right\}} \frac{k^{s-1} \Log^{r(r - p)/[p(r - 1)]} k \, n^{r/p}}{\mathrm{e}^{(r/p)k^{s}} (\LLog \, n)^{r^{2}/[p(r - 1)]}} \mathbb{P} \left\{\lvert X \rvert > \frac{n^{1/p}}{(\LLog \, n)^{r/[p(r - 1)]}} \right\} \\
&\quad \leqslant C(p,r) \sum_{n=1}^{\infty} \frac{(\LLog \, n)^{r(r - p)/[p(r - 1)]} n^{r/p}}{n^{r/p} (\LLog \, n)^{r^{2}/[p(r - 1)]}} \mathbb{P} \left\{\lvert X \rvert > \frac{n^{1/p}}{(\LLog \, n)^{r/[p(r - 1)]}} \right\} \\
&\quad \leqslant C(p,r) \sum_{n=1}^{\infty} \frac{1}{(\LLog \, n)^{r/(r - 1)}} \mathbb{P} \left\{\lvert X \rvert^{p} > \frac{n}{(\LLog \, n)^{r/(r - 1)}} \right\} < \infty
\end{align*}
which establishes \eqref{eq:3.11}. Similarly, \eqref{eq:3.12} and \eqref{eq:3.13} both hold because
\begin{align*}
&\sum_{k=1}^{\infty} \sum_{n=1}^{\ell_{k+1}} \sum_{j=\ell_{k}+1}^{\ell_{k+1}} \frac{\Log^{r} k \, n^{r/p} \Log^{-r^{2}/(r - p)} n}{j \ell_{k}^{r/p} (\LLog \, \ell_{k})^{r^{2}(p - 1)/[p(r - 1)]}} \mathbb{P} \left\{\lvert \varphi(X_{n}) \rvert > \frac{n^{1/p}}{\Log^{r/(r - p)} n} \right\} \\
&\quad \leqslant \sum_{k=1}^{\infty} \sum_{n=1}^{\ell_{k+1}} \frac{(\ell_{k+1} - \ell_{k}) \Log^{r} k \, n^{r/p} \Log^{-r^{2}/(r - p)} n}{\ell_{k}^{r/p + 1} (\LLog \, \ell_{k})^{r^{2}(p - 1)/[p(r - 1)]}} \mathbb{P} \left\{\lvert \varphi(X_{n}) \rvert > \frac{n^{1/p}}{\Log^{r/(r - p)} n} \right\} \\
&\quad \leqslant \sum_{k=1}^{\infty} \sum_{n=1}^{\ell_{k+1}} \frac{(\ell_{k+1} - \ell_{k}) \Log^{r} k \, n^{r/p} \Log^{-r^{2}/(r - p)} n}{\ell_{k}^{r/p + 1} (\LLog \, \ell_{k})^{r^{2}(p - 1)/[p(r - 1)]}} \mathbb{P} \left\{\lvert X_{n} \rvert > \frac{n^{1/p}}{\Log^{r/(r - p)} n} \right\} \\
&\quad \leqslant C \sum_{k=1}^{\infty} \sum_{n=1}^{\ell_{k+1}} \frac{(\ell_{k+1} - \ell_{k}) \Log^{r} k \, n^{r/p} \Log^{-r^{2}/(r - p)} n}{\ell_{k}^{r/p + 1} (\LLog \, \ell_{k})^{r^{2}(p - 1)/[p(r - 1)]}} \mathbb{P} \left\{\lvert X \rvert > \frac{n^{1/p}}{\Log^{r/(r - p)} n} \right\} \\
&\quad \leqslant C(p,r,s) \sum_{n=1}^{\infty} \sum_{\left\{k\colon \ell_{k+1} \geqslant n \right\}} \frac{k^{s-1} \Log^{r(r - p)/[p(r - 1)]} k \, n^{r/p}}{\mathrm{e}^{(r/p)k^{s}} \Log^{r^{2}/(r - p)} n} \mathbb{P} \left\{\lvert X \rvert > \frac{n^{1/p}}{\Log^{r/(r - p)} n} \right\} \\
&\quad \leqslant C(p,r,s) \sum_{n=1}^{\infty} \frac{(\LLog \, n)^{r(r - p)/[p(r - 1)]}}{\Log^{r^{2}/(r - p)} n} \mathbb{P} \left\{\lvert X \rvert^{p} > \frac{n}{\Log^{rp/(r - p)} n} \right\} \\
&\quad \leqslant C(p,r,s) \sum_{n=1}^{\infty} \frac{1}{\Log^{rp/(r - p)} n} \mathbb{P} \left\{\lvert X \rvert^{p} > \frac{n}{\Log^{rp/(r - p)} n} \right\} < \infty
\end{align*}
and
\begin{align*}
&\sum_{k=1}^{\infty} \sum_{n=1}^{\ell_{k+1}} \sum_{j=\ell_{k}+1}^{\ell_{k+1}} \frac{\Log^{r} k \, n^{r/p} (\LLog \, n)^{-r^{2}/[p(r - 1)]}}{j \ell_{k}^{r/p} (\LLog \, \ell_{k})^{r^{2}(p - 1)/[p(r - 1)]}} \mathbb{P} \left\{\lvert \varphi(X_{n}) \rvert > \frac{n^{1/p}}{(\LLog \, n)^{r/[p(r - 1)]}} \right\} \\
&\quad \leqslant \sum_{k=1}^{\infty} \sum_{n=1}^{\ell_{k+1}} \frac{(\ell_{k+1} - \ell_{k}) \Log^{r} k \, n^{r/p} (\LLog \, n)^{-r^{2}/[p(r - 1)]}}{\ell_{k}^{r/p + 1} (\LLog \, \ell_{k})^{r^{2}(p - 1)/[p(r - 1)]}} \mathbb{P} \left\{\lvert \varphi(X_{n}) \rvert > \frac{n^{1/p}}{(\LLog \, n)^{r/[p(r - 1)]}} \right\} \\
&\quad \leqslant \sum_{k=1}^{\infty} \sum_{n=1}^{\ell_{k+1}} \frac{(\ell_{k+1} - \ell_{k}) \Log^{r} k \, n^{r/p} (\LLog \, n)^{-r^{2}/[p(r - 1)]}}{\ell_{k}^{r/p + 1} (\LLog \, \ell_{k})^{r^{2}(p - 1)/[p(r - 1)]}} \mathbb{P} \left\{\lvert X_{n} \rvert > \frac{n^{1/p}}{(\LLog \, n)^{r/[p(r - 1)]}} \right\} \\
&\quad \leqslant C \sum_{k=1}^{\infty} \sum_{n=1}^{\ell_{k+1}} \frac{(\ell_{k+1} - \ell_{k}) \Log^{r} k \, n^{r/p} (\LLog \, n)^{-r^{2}/[p(r - 1)]}}{\ell_{k}^{r/p + 1} (\LLog \, \ell_{k})^{r^{2}(p - 1)/[p(r - 1)]}} \mathbb{P} \left\{\lvert X \rvert > \frac{n^{1/p}}{(\LLog \, n)^{r/[p(r - 1)]}} \right\} \\
&\quad \leqslant C(p,r,s) \sum_{n=1}^{\infty} \sum_{\left\{k\colon \ell_{k+1} \geqslant n \right\}} \frac{k^{s-1} \Log^{r(r - p)/[p(r - 1)]} k \, n^{r/p}}{\mathrm{e}^{(r/p)k^{s}} (\LLog \, n)^{r^{2}/[p(r - 1)]}} \mathbb{P} \left\{\lvert X \rvert > \frac{n^{1/p}}{(\LLog \, n)^{r/[p(r - 1)]}} \right\} \\
&\quad \leqslant C(p,r,s) \sum_{n=1}^{\infty} \frac{1}{(\LLog \, n)^{r/(r - 1)}} \mathbb{P} \left\{\lvert X \rvert^{p} > \frac{n}{(\LLog \, n)^{r/(r - 1)}} \right\} < \infty
\end{align*}
by using \eqref{eq:3.18} and Lemma 3 of \cite{Lita15}. It remains to prove \eqref{eq:3.14}. According to Lemma 1 of \cite{Lita15}, we get
\begin{align*}
&\sum_{j=\ell_{k}+1}^{\ell_{k+1}} \frac{\mathbb{E} \lvert \varphi(X_{j}) \rvert I_{\left\{\lvert \varphi(X_{j}) \rvert > j^{1/p}/\Log^{r/(r - p)} j \right\}}}{\ell_{k}^{1/p} (\LLog \, \ell_{k})^{r(p - 1)/[p(r - 1)]}} \\
&\quad \leqslant \sum_{j=\ell_{k}+1}^{\ell_{k+1}} \frac{\mathbb{E} \lvert X_{j} \rvert I_{\left\{\lvert X_{j} \rvert > j^{1/p}/\Log^{r/(r - p)} j \right\}}}{\ell_{k}^{1/p} (\LLog \, \ell_{k})^{r(p - 1)/[p(r - 1)]}} \\
&\quad \leqslant C \sum_{j=\ell_{k}+1}^{\ell_{k+1}} \frac{\mathbb{E} \lvert X \rvert I_{\left\{\lvert X \rvert > j^{1/p}/\Log^{r/(r - p)} j \right\}}}{\ell_{k}^{1/p} (\LLog \, \ell_{k})^{r(p - 1)/[p(r - 1)]}} \\
&\quad \leqslant C \frac{\ell_{k+1} - \ell_{k}}{\ell_{k}^{1/p} (\LLog \, \ell_{k})^{r(p - 1)/[p(r - 1)]}} \mathbb{E} \lvert X \rvert I_{\left\{\lvert X \rvert > \ell_{k}^{1/p}/\Log^{r/(r - p)} \ell_{k} \right\}} \\
&\quad \leqslant C \frac{(\ell_{k+1} - \ell_{k}) \ell_{k}^{(1 - p)/p} \Log^{r(p - 1)/(r - p)} \, \ell_{k}}{\ell_{k}^{1/p} (\LLog \, \ell_{k})^{r(p - 1)/[p(r - 1)]}} \mathbb{E} \lvert X \rvert^{p} I_{\left\{\lvert X \rvert > \ell_{k}^{1/p}/\Log^{r/(r - p)} \ell_{k} \right\}} \\
&\quad \leqslant C(p,r,s) \frac{k^{s - 1 + sr(p - 1)/(r - p)}}{\Log^{r(p - 1)/[p(r - 1)]}k} \mathbb{E} \lvert X \rvert^{p} I_{\left\{\lvert X \rvert > \ell_{k}^{1/p}/\Log^{r/(r - p)} \ell_{k} \right\}}
\end{align*}
for the reason that $\Log \, \ell_{k} \sim k^{s} $ as $k \rightarrow \infty$.
\end{proof}

\begin{proofTheorem1}
Setting
\begin{equation}\label{eq:3.19}
\begin{gathered}
X_{n}' := X_{n} I_{\{0 \leqslant X_{n} \leqslant c_{n} \}} + c_{n} I_{\{X_{n} > c_{n} \}}, \\
X_{n}'' := X_{n} I_{\{c_{n} < X_{n} \leqslant d_{n} \}} - c_{n} I_{\{X_{n} > c_{n} \}} + d_{n} I_{\{X_{n} > d_{n} \}}, \\
X_{n}''' := X_{n} I_{\{X_{n} > d_{n} \}} - d_{n} I_{\{X_{n} > d_{n} \}}
\end{gathered}
\end{equation}
we have $X_{n}' + X_{n}'' + X_{n}''' = X_{n}$. Hence, assumptions (c), (d) guarantee for all $\varepsilon > 0$,
\begin{equation*}
\begin{split}
&\sum_{n=1}^{\infty} a_{n} \mathbb{P} \left\{\max_{1 \leqslant j \leqslant n} \abs{\sum_{k=1}^{j} (X_{k}' - \mathbb{E} \, X_{k}')} > \varepsilon b_{n} \right\} \\
&\quad \leqslant \sum_{n=1}^{\infty} \frac{a_{n}}{\varepsilon^{r} b_{n}^{r}} \mathbb{E} \left[\max_{1 \leqslant j \leqslant n} \abs{\sum_{k=1}^{j} (X_{k}' - \mathbb{E} \, X_{k}')} \right]^{r} \\
&\quad \leqslant \frac{1}{\varepsilon^{r}} \sum_{n=1}^{\infty} \frac{a_{n} \Lambda_{n}^{r}}{b_{n}^{r}} \sum_{k=1}^{n} \mathbb{E} \, \lvert X_{k}' - \mathbb{E} \, X_{k}' \rvert^{r} \qquad (\text{by \eqref{eq:2.2} with $\xi_{k} = k$ and $\eta = 0$}) \\
&\quad \leqslant \frac{C(r)}{\varepsilon^{r}} \sum_{n=1}^{\infty} \frac{a_{n} \Lambda_{n}^{r}}{b_{n}^{r}} \sum_{k=1}^{n} \big(\mathbb{E} \, X_{k}^{r} I_{\{X_{k} \leqslant c_{k} \}}  + c_{k}^{r} \mathbb{P} \{X_{k} > c_{k} \} \big) \\
&\quad = \frac{C(r)}{\varepsilon^{r}} \sum_{n=1}^{\infty} \sum_{k=1}^{n} \frac{a_{n} \Lambda_{n}^{r}}{b_{n}^{r}} \mathbb{E} \, X_{k}^{r} I_{\{X_{k} \leqslant c_{k} \}}  + \frac{C(r)}{\varepsilon^{r}} \sum_{n=1}^{\infty} \sum_{k=1}^{n} \frac{a_{n} \Lambda_{n}^{r} c_{k}^{r}}{b_{n}^{r}} \mathbb{P} \{X_{k} > c_{k} \} < \infty,
\end{split}
\end{equation*}
and Theorem 2.2 of \cite{Hu16} yields
\begin{equation}\label{eq:3.20}
\frac{1}{b_{n}} \sum_{k=1}^{n} (X_{k}' - \mathbb{E} \, X_{k}') \overset{\textnormal{a.s.}}{\longrightarrow} 0.
\end{equation}
Moreover, from Theorem 1 of \cite{Kounias69} and condition (e) we still have
\begin{equation*}
\begin{split}
&\sum_{n=1}^{\infty} a_{n} \mathbb{P} \left\{\max_{1 \leqslant j \leqslant n} \abs{\sum_{k=1}^{j} (X_{k}''' - \mathbb{E} \, X_{k}''')} > \varepsilon b_{n} \right\} \\
&\qquad \leqslant \sum_{n=1}^{\infty} \frac{a_{n}}{\varepsilon b_{n}} \sum_{k=1}^{n} \mathbb{E} \, \lvert X_{k}''' - \mathbb{E} \, X_{k}''' \rvert \\
&\qquad \leqslant \frac{2}{\varepsilon} \sum_{n=1}^{\infty} \sum_{k=1}^{n} \frac{a_{n}}{b_{n}} \mathbb{E} \, X_{k} I_{\{ X_{k} > d_{k} \}} \\
&\qquad < \infty.
\end{split}
\end{equation*}
which leads to
\begin{equation}\label{eq:3.21}
\frac{1}{b_{n}} \sum_{k=1}^{n} (X_{k}''' - \mathbb{E} \, X_{k}''') \overset{\textnormal{a.s.}}{\longrightarrow} 0
\end{equation}
according to Theorem 2.2 of \cite{Hu16}. Supposing
\begin{equation*}
T_{k} := \sum_{j=\ell_{k-1} + 1}^{\ell_{k}} (X_{j}'' - \mathbb{E} \, X_{j}'')
\end{equation*}
and $\ell_{0} := 0$, we get
\begin{align*}
&\sum_{k=1}^{\infty} \mathbb{P} \left\{\max_{1 \leqslant n \leqslant k+1} \abs{\sum_{i=1}^{\ell_{n}} (X_{i}'' - \mathbb{E} \, X_{i}'')} > \varepsilon b_{\ell_{k}}  \right\} \sum_{j=\ell_{k}+1}^{\ell_{k+1}} a_{j} \\
&\quad =\sum_{k=1}^{\infty} \mathbb{P} \left\{\max_{1 \leqslant n \leqslant k+1} \abs{\sum_{i=1}^{n} (T_{i} - \mathbb{E} \, T_{i})} > \varepsilon b_{\ell_{k}} \right\} \sum_{j=\ell_{k}+1}^{\ell_{k+1}} a_{j} \\
&\quad \leqslant \frac{1}{\varepsilon^{r}} \sum_{k=1}^{\infty} \sum_{j=\ell_{k}+1}^{\ell_{k+1}} \frac{a_{j}}{b_{\ell_{k}}^{r}} \, \mathbb{E} \left[\max_{1 \leqslant n \leqslant k+1} \abs{\sum_{i=1}^{n} (T_{i} - \mathbb{E} \, T_{i})} \right]^{r} \\
&\quad \leqslant \frac{1}{\varepsilon^{r}} \sum_{k=1}^{\infty} \sum_{j=\ell_{k}+1}^{\ell_{k+1}} \frac{a_{j} \Lambda_{k+1}^{r}}{b_{\ell_{k}}^{r}} \sum_{i=1}^{k+1} \mathbb{E} \, \lvert T_{i} - \mathbb{E} \, T_{i} \rvert^{r} \qquad (\text{by \eqref{eq:2.2} taking $\xi_{k} = \ell_{k}$ and $\eta = 0$}) \\
&\quad \leqslant \frac{C(r)}{\varepsilon^{r}} \sum_{k=1}^{\infty} \sum_{j=\ell_{k}+1}^{\ell_{k+1}} \frac{a_{j} \Lambda_{k+1}^{r}}{b_{\ell_{k}}^{r}} \sum_{i=1}^{k+1} \mathbb{E} \, \lvert T_{i} \rvert^{r} \\
&\quad \leqslant \frac{C(r)}{\varepsilon^{r}} \sum_{k=1}^{\infty} \sum_{j=\ell_{k}+1}^{\ell_{k+1}} \frac{a_{j} \Lambda_{k+1}^{r} \lambda_{\ell_{k} - \ell_{k-1}}^{r}}{b_{\ell_{k}}^{r}} \sum_{i=1}^{\ell_{k+1}} \mathbb{E} \, \lvert X_{i}'' - \mathbb{E} \, X_{i}'' \rvert^{r} \qquad \left(\substack{\text{by \eqref{eq:2.1} with $\xi_{k} = k$,} \\ \text{$\eta = \ell_{k-1}$ and $n = \ell_{k} - \ell_{k-1}$}} \right) \\
&\quad \leqslant \frac{C(r)}{\varepsilon^{r}} \sum_{k=1}^{\infty}  \sum_{j=\ell_{k}+1}^{\ell_{k+1}} \frac{a_{j} \Lambda_{k+1}^{r} \lambda_{\ell_{k} - \ell_{k-1}}^{r}}{b_{\ell_{k}}^{r}} \sum_{i=1}^{\ell_{k+1}} \mathbb{E} \, \lvert X_{i}'' \rvert^{r} \\
&\quad \leqslant \frac{C(r)}{\varepsilon^{r}} \sum_{k=1}^{\infty}  \sum_{j=\ell_{k}+1}^{\ell_{k+1}} \frac{a_{j} \Lambda_{k+1}^{r} \lambda_{\ell_{k} - \ell_{k-1}}^{r}}{b_{\ell_{k}}^{r}} \sum_{i=1}^{\ell_{k+1}} \big(\mathbb{E} \, X_{i}^{r} I_{\{c_{i} < X_{i} \leqslant d_{i} \}} + c_{i}^{r} \mathbb{P} \{X_{i} > c_{i} \} + d_{i}^{r} \mathbb{P} \{X_{i} > d_{i} \}\big) \\
&\quad \leqslant \frac{C(r)}{\varepsilon^{r}} \sum_{k=1}^{\infty} \sum_{j=\ell_{k}+1}^{\ell_{k+1}} \sum_{i=1}^{\ell_{k+1}} \frac{a_{j} \Lambda_{k+1}^{r} \lambda_{\ell_{k} - \ell_{k-1}}^{r}}{b_{\ell_{k}}^{r}} \mathbb{E} \, X_{i}^{r} I_{\{c_{i} < X_{i} \leqslant d_{i} \}} \\
& \qquad + \frac{C(r)}{\varepsilon^{r}} \sum_{k=1}^{\infty} \sum_{j=\ell_{k}+1}^{\ell_{k+1}} \sum_{i=1}^{\ell_{k+1}} \left(\frac{a_{j} \Lambda_{k+1}^{r} c_{i}^{r} \lambda_{\ell_{k} - \ell_{k-1}}^{r}}{b_{\ell_{k}}^{r}} \mathbb{P} \{X_{i} > c_{i} \} + \frac{a_{j} \Lambda_{k+1}^{r} d_{i}^{r} \lambda_{\ell_{k} - \ell_{k-1}}^{r}}{b_{\ell_{k}}^{r}} \mathbb{P} \{X_{i} > d_{i} \} \right) \\
& \quad < \infty
\end{align*}
employing assumptions (f), (g), (h). Further,
\begin{gather*}
\frac{1}{b_{\ell_{k}}} \sum_{j=\ell_{k}+1}^{\ell_{k+1}} \mathbb{E} \, X_{j} I_{\left\{c_{j} < X_{j} \leqslant d_{j} \right\}} \leqslant \frac{1}{b_{\ell_{k}}} \sum_{j=\ell_{k}+1}^{\ell_{k+1}} \mathbb{E} \, X_{j} I_{\left\{X_{j} > c_{j} \right\}}, \\
\frac{1}{b_{\ell_{k}}} \sum_{j=\ell_{k}+1}^{\ell_{k+1}} d_{j} \mathbb{P} \big\{X_{j} > d_{j} \big\} \leqslant \frac{1}{b_{\ell_{k}}} \sum_{j=\ell_{k}+1}^{\ell_{k+1}} \mathbb{E} \, X_{j} I_{\left\{X_{j} > d_{j} \right\}} \leqslant \frac{1}{b_{\ell_{k}}} \sum_{j=\ell_{k}+1}^{\ell_{k+1}} \mathbb{E} \, X_{j} I_{\left\{X_{j} > c_{j} \right\}}
\end{gather*}
and
\begin{equation*}
0 \leqslant \frac{1}{b_{\ell_{k}}} \sum_{j=\ell_{k}+1}^{\ell_{k+1}} \mathbb{E} \, X_{j}'' \leqslant \frac{1}{b_{\ell_{k}}} \sum_{j=\ell_{k}+1}^{\ell_{k+1}} \mathbb{E} \, X_{j} I_{\left\{c_{j} < X_{j} \leqslant d_{j} \right\}} + \frac{1}{b_{\ell_{k}}} \sum_{j=\ell_{k}+1}^{\ell_{k+1}} d_{j} \mathbb{P} \big\{X_{j} > d_{j} \big\}
\end{equation*}
so that, condition (b) entails $\sum_{j=\ell_{k}+1}^{\ell_{k+1}} \mathbb{E} \, X_{j}''/b_{\ell_{k}} = o(1)$ as $k \rightarrow \infty$. From Lemma~\ref{lem:1} we get
\begin{equation}\label{eq:3.22}
\frac{1}{b_{n}} \sum_{k=1}^{n} (X_{k}'' - \mathbb{E} \, X_{k}'') \overset{\textnormal{a.s.}}{\longrightarrow} 0.
\end{equation}
Thus, \eqref{eq:3.20}, \eqref{eq:3.21} and \eqref{eq:3.22} ensure
\begin{equation*}
\frac{1}{b_{n}} \sum_{k=1}^{n} (X_{k} - \mathbb{E} \, X_{k}) = \frac{1}{b_{n}} \sum_{k=1}^{n} (X_{k}' - \mathbb{E} \, X_{k}') + \frac{1}{b_{n}} \sum_{k=1}^{n} (X_{k}'' - \mathbb{E} \, X_{k}'') + \frac{1}{b_{n}} \sum_{k=1}^{n} (X_{k}''' - \mathbb{E} \, X_{k}''') \overset{\textnormal{a.s.}}{\longrightarrow} 0
\end{equation*}
establishing the thesis.
\end{proofTheorem1}

\begin{proofCorollary}
By noting that $X_{n} = X_{n}^{+} - X_{n}^{-}$, where $X_{n}^{+} := \max(X_{n},0) \geqslant 0$ and $X_{n}^{-} := \max(-X_{n},0) \geqslant 0$, we have
\begin{equation*}
\frac{1}{b_{n}} \sum_{k=1}^{n} (X_{k} - \mathbb{E} \, X_{k}) = \frac{1}{b_{n}} \sum_{k=1}^{n} (X_{k}^{+} - \mathbb{E} \, X_{k}^{+}) - \frac{1}{b_{n}} \sum_{k=1}^{n} (X_{k}^{-} - \mathbb{E} \, X_{k}^{-}).
\end{equation*}
It remains to show
\begin{equation*}
\frac{1}{b_{n}} \sum_{k=1}^{n} (X_{k}^{+} - \mathbb{E} \, X_{k}^{+}) \overset{\textnormal{a.s.}}{\longrightarrow} 0 \quad \text{and} \quad \frac{1}{b_{n}} \sum_{k=1}^{n} (X_{k}^{-} - \mathbb{E} \, X_{k}^{-}) \overset{\textnormal{a.s.}}{\longrightarrow} 0.
\end{equation*}
Since prior assertions can be proven in the same way, we only prove
\begin{equation}\label{eq:3.23}
\frac{1}{b_{n}} \sum_{k=1}^{n} (X_{k}^{+} - \mathbb{E} \, X_{k}^{+}) \overset{\textnormal{a.s.}}{\longrightarrow} 0.
\end{equation}
The sequence of random variables $\{X_{n}^{+}, \, n \geqslant 1\}$ is pairwise NQD because $x \mapsto \max(x,0)$ is a nondecreasing function (see \cite{Lehmann66}). Let $\{\xi_{k} \}$ be an increasing sequence of nonnegative integers and $S_{k}^{+} := \sum_{j=\xi_{k-1}+1}^{\xi_{k}} g_{s,t}(X_{j}^{+})$, $t > s \geqslant 0$. We have, for all $\eta \geqslant 0$ and $n \geqslant 1$,
\begin{align*}
&\mathbb{E} \left\{\sum_{k=\eta + 1}^{\eta + n} \sum_{j=\xi_{k-1}+1}^{\xi_{k}} \big[g_{s,t}(X_{j}^{+}) - \mathbb{E} \,  g_{s,t}(X_{j}^{+}) \big] \right\}^{2} \\
&\quad = \mathbb{E} \left[\sum_{k=\eta + 1}^{\eta + n} (S_{k}^{+} - \mathbb{E} \,  S_{k}^{+}) \right]^{2} \\
&\quad \leqslant \sum_{k=\eta + 1}^{\eta + n} \mathbb{E} (S_{k}^{+} - \mathbb{E} \, S_{k}^{+})^{2} + 2 \sum_{\substack{k < j}} \mathbb{E}\big[(S_{k}^{+} - \mathbb{E} \, S_{k}^{+}) (S_{j}^{+} - \mathbb{E} \, S_{j}^{+}) \big] \\
&\quad = \sum_{k=\eta + 1}^{\eta + n} \mathbb{E} (S_{k}^{+} - \mathbb{E} \, S_{k}^{+})^{2} \\
&\qquad + 2 \sum_{\substack{k < j}} \mathbb{E} \left\{\sum_{p=\xi_{k-1}+1}^{\xi_{k}} \sum_{q=\xi_{j-1}+1}^{\xi_{j}} \big[g_{s,t}(X_{p}^{+}) - \mathbb{E} \, g_{s,t}(X_{p}^{+}) \big] \big[g_{s,t}(X_{q}^{+}) - \mathbb{E} \, g_{s,t}(X_{q}^{+}) \big] \right\} \\
&\quad = \sum_{k=\eta + 1}^{\eta + n} \mathbb{E} (S_{k}^{+} - \mathbb{E} \, S_{k}^{+})^{2} \\
&\qquad + 2 \sum_{\substack{k < j}} \sum_{p=\xi_{k-1}+1}^{\xi_{k}} \sum_{q=\xi_{j-1}+1}^{\xi_{j}} \mathbb{E} \big[g_{s,t}(X_{p}^{+}) - \mathbb{E} \, g_{s,t}(X_{p}^{+}) \big] \big[g_{s,t}(X_{q}^{+}) - \mathbb{E} \, g_{s,t}(X_{q}^{+}) \big] \\
&\quad \leqslant \sum_{k=\eta + 1}^{\eta + n} \mathbb{E} (S_{k}^{+} - \mathbb{E} \, S_{k}^{+})^{2} \\
&\qquad + 2 \sum_{\substack{k < j}} \sum_{p=\xi_{k-1}+1}^{\xi_{k}} \sum_{q=\xi_{j-1}+1}^{\xi_{j}} \mathbb{E} \big[g_{s,t}(X_{p}^{+}) - \mathbb{E} \, g_{s,t}(X_{p}^{+}) \big] \mathbb{E} \big[g_{s,t}(X_{q}^{+}) - \mathbb{E} \, g_{s,t}(X_{q}^{+}) \big] \\
&\quad \leqslant \sum_{k=\eta + 1}^{\eta + n} \mathbb{E} \left\{\sum_{j=\xi_{k-1}+1}^{\xi_{k}} \big[g_{s,t}(X_{j}^{+}) - \mathbb{E} \,  g_{s,t}(X_{j}^{+}) \big] \right\}^{2}
\end{align*}
by noting that $\mathbb{E}(X Y) \leqslant \mathbb{E} \, X \, \mathbb{E} \, Y$ for any NQD random variables $X,Y$ and $\{g_{s,t}(X_{n}^{+}) - \mathbb{E} \, g_{s,t}(X_{n}^{+}), \, n \geqslant 1\}$ is still a sequence of pairwise NQD random variables because $x \mapsto g_{s,t}(x) - C(s,t)$ is a nondecreasing function. Thus, $\{X_{n}^{+}, \, n \geqslant 1\}$ satisfies \eqref{eq:2.1} with $r = 2$ and $\lambda_{n} = 1$ for all $n$. Hence, $\Lambda_{n} \leqslant \log(2n)$ and considering $\ell_{k} = \big\lfloor\mathrm{e}^{{k}^{(2 - p)/p}} \big\rfloor$, $a_{n} = 1/n$, $b_{n} = n^{1/p} (\LLog \, n)^{2(p - 1)/p}$, $c_{n} = n^{1/p}/(\Log \, n)^{2/(2 - p)}$, $d_{n} = n^{1/p}/(\LLog \, n)^{2/p}$, $m_{k} = 2^{k}$, we obtain ${\displaystyle \limsup_{k \rightarrow \infty}} \, b_{m_{k+1}}/b_{m_{k}} = 2^{1/p}$,
\begin{equation*}
\liminf_{k \rightarrow \infty} \sum_{j=m_{k}}^{m_{k+1} - 1} \frac{1}{j} \geqslant \liminf_{k \rightarrow \infty} \int_{2^{k}}^{2^{k+1}} \frac{\mathrm{d}x}{x} = \log 2
\end{equation*}
which shows that condition (a) of Theorem~\ref{thr:1} holds. 
The remaining assumptions of Theorem~\ref{thr:1} are a consequence of Lemmas~\ref{lem:2} and~\ref{lem:4}. The proof is complete.
\end{proofCorollary}

\begin{acknowledgements}
This work is a contribution to the Project UIDB/04035/2020, funded by FCT - Funda\c{c}\~{a}o para a Ci\^{e}ncia e a Tecnologia, Portugal.
\end{acknowledgements}

\end{document}